\documentclass[12pt]{amsart}
\usepackage{amssymb}

\newtheorem{theorem}{Theorem}[section]
\newtheorem{definition}[theorem]{Definition}
\newtheorem{proposition}[theorem]{Proposition}
\newtheorem{lemma}[theorem]{Lemma}

\newtheorem{corollary}[theorem]{Corollary}

\begin{document}

\title{A note on free quantum groups}
\author{Teodor Banica}
\address{Department of Mathematics,
Paul Sabatier University, 118 route de Narbonne, 31062 Toulouse, France} 
\email{banica@picard.ups-tlse.fr}
\subjclass[2000]{16W30} 
\keywords{Free quantum group}
\thanks{Work supported by the CNRS and by the Fields Institute}

\begin{abstract}
We study the free complexification operation for compact quantum groups, $G\to G^c$. We prove that, with suitable definitions, this induces a one-to-one correspondence between free orthogonal quantum groups of infinite level, and free unitary quantum groups satisfying $G=G^c$.
\end{abstract}

\maketitle

\section*{Introduction}

In this paper we present some advances on the notion of free quantum group, introduced in \cite{bbc}. We first discuss in detail a result mentioned there, namely that the free complexification operation $G\to G^c$ studied in \cite{b2} produces free unitary quantum groups out of free orthogonal ones. Then we work out the injectivity and surjectivity properties of $G\to G^c$, and this leads to the correspondence announced in the abstract. This correspondence should be regarded as being a first general ingredient for the classification of free quantum groups.

We include in our study a number of general facts regarding the operation $G\to G^c$, by improving some previous work in \cite{b2}. The point is that now we can use general diagrammatic techniques from \cite{bc}, new examples, and the notion of free quantum group \cite{bbc}, none of them available at the time of writing \cite{b2}.

The paper is organized as follows: 1 contains some basic facts about the operation $G\to G^c$, and in 2-5 we discuss the applications to free quantum groups.

\section{Free complexification}

A fundamental result of Voiculescu \cite{vo} states if $(s_1,\ldots,s_n)$ is a semicircular system, and $z$ is a Haar unitary free from it, then $(zs_1,\ldots,zs_n)$ is a circular system. This makes appear the notion of free multiplication by a Haar unitary, $a\to za$, that we call here free complexification. This operation has been intensively studied since then. See Nica and Speicher \cite{ns}.

This operation appears as well in the context of Wang's free quantum groups \cite{wa1}, \cite{wa2}. The main result in \cite{b1} is that the universal free biunitary matrix is the free complexification of the free orthogonal matrix. In other words, the passage $O_n^+\to U_n^+$ is nothing but a free complexification: $U_n^+=O_n^{+c}$. Moreover, some generalizations of this fact are obtained, in an abstract setting, in \cite{b2}.

In this section we discuss the basic properties of $A\to\tilde{A}$, the functional analytic version of $G\to G^c$. We use an adaptation of Woronowicz's axioms in \cite{wo1}.

\begin{definition}
A finitely generated Hopf algebra is a pair $(A,u)$, where $A$ is a $C^*$-algebra and $u\in M_n(A)$ is a unitary whose entries generate $A$, such that
\begin{eqnarray*}
\Delta(u_{ij})&=&\sum u_{ik}\otimes u_{kj}\\
\varepsilon(u_{ij})&=&\delta_{ij}\\
S(u_{ij})&=&u_{ji}^*
\end{eqnarray*}
define morphisms of $C^*$-algebras (called comultiplication, counit and antipode).
\end{definition}

In other words, given $(A,u)$, the morphisms $\Delta,\varepsilon,S$ can exist or not. If they exist, they are uniquely determined, and we say that we have a Hopf algebra.

The basic examples are as follows:
\begin{enumerate}
\item The algebra of functions $A=C(G)$, with the matrix $u=(u_{ij})$ given by $g=(u_{ij}(g))$, where $G\subset U_n$ is a compact group.

\item The group algebra $A=C^*(\Gamma)$, with the matrix $u={\rm diag}(g_1,\ldots,g_n)$, where $\Gamma=<g_1,\ldots,g_n>$ is a finitely generated group.
\end{enumerate}

Let $\mathbb T$ be the unit circle, and let $z:\mathbb T\to\mathbb C$ be the identity function, $z(x)=x$. Observe that $(C(\mathbb T),z)$ is a finitely generated Hopf algebra, corresponding to the compact group $\mathbb T\subset U_1$, or, via the Fourier transform, to the group $\mathbb Z=<1>$.

\begin{definition}
Associated to $(A,u)$ is the pair $(\tilde{A},\tilde{u})$, where $\tilde{A}\subset C(\mathbb T)*A$ is the $C^*$-algebra generated by the entries of the matrix $\tilde{u}=zu$.
\end{definition}

It follows from the general results of Wang in \cite{wa1} that $(\tilde{A},\tilde{u})$ is indeed a finitely generated Hopf algebra. Moreover, $\tilde{u}$ is the free complexification of $u$ in the free probabilistic sense, i.e. with respect to the Haar functional. See \cite{b2}.

A morphism between two finitely generated Hopf algebras $f:(A,u)\to (B,v)$ is by definition a morphism of $*$-algebras $A_s\to B_s$ mapping $u_{ij}\to v_{ij}$, where $A_s\subset A$ and $B_s\subset B$ are the dense $*$-subalgebras generated by the elements $u_{ij}$, respectively $v_{ij}$. Observe that in order for a such a morphism to exist, $u,v$ must have the same size, and that if such a morphism exists, it is unique. See \cite{b2}.

\begin{proposition}
The operation $A\to\tilde{A}$ has the following properties:
\begin{enumerate}
\item We have a morphism $(\tilde{A},\tilde{u})\to (A,u)$.
\item A morphism $(A,u)\to(B,v)$ produces a morphism $(\tilde{A},\tilde{u})\to(\tilde{B},\tilde{v})$.
\item We have an isomorphism $(\tilde{\tilde{A}},\tilde{\tilde{u}})=(\tilde{A},\tilde{u})$.
\end{enumerate}
\end{proposition}

\begin{proof}
All the assertions are clear from definitions, see \cite{b2} for details.
\end{proof}

\begin{theorem}
If $\Gamma=<g_1,\ldots,g_n>$ is a finitely generated group then $\tilde{C}^*(\Gamma)\simeq C^*(\mathbb Z*\Lambda)$, where $\Lambda=<g_i^{-1}g_j\ |\ i,j=1,\ldots,n>$.
\end{theorem}

\begin{proof}
By using the Fourier transform isomorphism $C(\mathbb T)\simeq C^*(\mathbb Z)$ we obtain $\tilde{C}^*(\Gamma)=C^*(\tilde{\Gamma})$, with $\tilde{\Gamma}\subset\mathbb Z*\Gamma$. Then, a careful examination of generators gives the isomorphism $\tilde{\Gamma}\simeq\mathbb Z*\Lambda$. See \cite{b2} for details.
\end{proof}

At the dual level, we have the following question: what is the compact quantum group $G^c$ defined by $C(G^c)=\tilde{C}(G)$? There is no simple answer to this question, unless in the abelian case, where we have the following result.

\begin{theorem}
If $G\subset U_n$ is a compact abelian group then $\tilde{C}(G)=C^*(\mathbb Z*\widehat{L})$, where $L$ is the image of $G$ in the projective unitary group $PU_n$.
\end{theorem}

\begin{proof}
The embedding $G\subset U_n$, viewed as a representation, must come from a generating system $\widehat{G}=<g_1,\ldots,g_n>$. It routine to check that the subgroup $\Lambda\subset\widehat{G}$ constructed in Theorem 1.4 is the dual of $L$, and this gives the result.
\end{proof}

\section{Free quantum groups}

Consider the groups $S_n\subset O_n\subset U_n$, with the elements of $S_n$ viewed as permutation matrices. Consider also the following subgroups of $U_n$:
\begin{enumerate}
\item $S_n'=\mathbb Z_2\times S_n$, the permutation matrices multiplied by $\pm 1$.
\item $H_n=\mathbb Z_2\wr S_n$, the permutation matrices with $\pm$ coefficients. 
\item $P_n=\mathbb T\times S_n$, the permutation matrices multiplied by scalars in $\mathbb T$.
\item $K_n=\mathbb T\wr S_n$, the permutation matrices with coefficients in $\mathbb T$. 
\end{enumerate}

Observe that $H_n$ is the hyperoctahedral group. It is convenient to collect the above definitions into a single one, in the following way.

\begin{definition}
We use the diagram of compact groups
$$\begin{matrix}
U_n&\supset&K_n&\supset&P_n\cr\cr
\cup&&\cup&&\cup\cr\cr
O_n&\supset&H_n&\supset&S_n^*
\end{matrix}$$
where $S^*$ denotes at the same time $S$ and $S'$.
\end{definition}

In what follows we describe the free analogues of these 7 groups. For this purpose, we recall that a square matrix $u\in M_n(A)$ is called:
\begin{enumerate}
\item Orthogonal, if $u=\bar{u}$ and $u^t=u^{-1}$. 
\item Cubic, if it is orthogonal, and $ab=0$ on rows and columns.
\item Magic', if it is cubic, and the sum on rows and columns is the same.
\item Magic, if it is cubic, formed of projections ($a^2=a=a^*$).
\item Biunitary, if both $u$ and $u^t$ are unitaries. 
\item Cubik, if it is biunitary, and $ab^*=a^*b=0$ on rows and columns.
\item Magik, if it is cubik, and the sum on rows and columns is the same.
\end{enumerate}

Here the equalities of type $ab=0$ refer to distinct entries on the same row, or on the same column. The notions (1, 2, 4, 5) are from \cite{wa1,bbc,wa2,wa1}, and (3, 6, 7) are new. The terminology is of course temporary: we have only 7 examples of free quantum groups, so we don't know exactly what the names name.

\begin{theorem}
$C(G_n)$ with $G=OHS^*UKP$ is the universal commutative $C^*$-algebra generated by the entries of a $n\times n$ orthogonal, cubic, magic*, biunitary, cubik, magik matrix.
\end{theorem}

\begin{proof}
The case $G=OHSU$ is discussed in \cite{wa1,bbc,wa2,wa1}, and the case $G=S'KP$ follows from it, by identifying the corresponding subgroups.
\end{proof}

We proceed with liberation: definitions will become theorems and vice versa.

\begin{definition}
$A_g(n)$ with $g=ohs^*ukp$ is the universal $C^*$-algebra generated by the entries of a $n\times n$ orthogonal, cubic, magic*, biunitary, cubik, magik matrix.
\end{definition}

The $g=ohsu$ algebras are from \cite{wa1,bbc,wa2,wa1}, and the $g=s'kp$ ones are new.

\begin{theorem}
We have the diagram of Hopf algebras
$$\begin{matrix}
A_u(n)&\to&A_k(n)&\to&A_p(n)\cr\cr
\downarrow&&\downarrow&&\downarrow\cr\cr
A_o(n)&\to&A_h(n)&\to&A_{s^*}(n)
\end{matrix}$$
where $s^*$ denotes at the same time $s$ and $s'$.
\end{theorem}

\begin{proof}
The morphisms in Definition 1.1 can be constructed by using the universal property of each of the algebras involved. For the algebras $A_{ohsu}$ this is known from \cite{wa1,bbc,wa2,wa1}, and for the algebras $A_{s'kp}$ the proof is similar.
\end{proof}

\section{Diagrams}

Let $F=<a,b>$ be the monoid of words on two letters $a,b$. For a given corepresentation $u$ we let $u^a=u$, $u^b=\bar{u}$, then we define the tensor powers $u^\alpha$ with $\alpha\in F$ arbitrary, according to the rule $u^{\alpha\beta}=u^\alpha\otimes u^\beta$.

\begin{definition}
Let $(A,u)$ be a finitely generated Hopf algebra.
\begin{enumerate}
\item $CA$ is the collection of linear spaces $\{Hom(u^\alpha,u^\beta)|\alpha,\beta\in F\}$.
\item In the case $u=\bar{u}$ we identify $CA$ with $\{Hom(u^k,u^l)|k,l\in\mathbb N\}$.
\end{enumerate}
\end{definition}

A morphism $(A,u)\to (B,v)$ produces inclusions $Hom(u^\alpha,u^\beta)\subset Hom(v^\alpha,v^\beta)$ for any $\alpha,\beta\in F$, so we have the following diagram:
$$\begin{matrix}
CA_u(n)&\subset&CA_k(n)&\subset&CA_p(n)\cr\cr
\cap&&\cap&&\cap\cr\cr
CA_o(n)&\subset&CA_h(n)&\subset&CA_{s^*}(n)
\end{matrix}$$

We recall that $CA_s(n)$ is the category of Temperley-Lieb diagrams. That is, $Hom(u^k,u^l)$ is isomorphic to the abstract vector space spanned by the diagrams between an upper row of $2k$ points, and a lower row of $2l$ points. See \cite{bbc}.

In order to distinguish between various meanings of the same diagram, we attach words to it. For instance $\Cap_{ab},\Cap_{ba}$ are respectively in $D_s(\emptyset,ab),D_s(\emptyset,ba)$.

\begin{lemma}
The categories for $A_g(n)$ with $g=ohsukp$ are as follows:
\begin{enumerate}
\item $CA_o(n)=<\Cap>$.
\item $CA_h(n)=<\Cap,|{\:}^\cup_\cap\,|>$.
\item $CA_{s'}(n)=<\Cap,|{\:}^\cup_\cap\,|,{{\;}^\cup_\cap}>$.
\item $CA_s(n)=<\Cap,|{\:}^\cup_\cap\,|,\cap>$.
\item $CA_u(n)=<\Cap_{ab},\Cap_{ba}>$.
\item $CA_k(n)=<\Cap_{ab},\Cap_{ba},|{\:}^\cup_\cap\,|^{ab}_{ab},|{\:}^\cup_\cap\,|^{ba}_{ba}>$.
\item $CA_p(n)=<\Cap_{ab},\Cap_{ba},|{\:}^\cup_\cap\,|^{ab}_{ab},|{\:}^\cup_\cap\,|^{ba}_{ba},{{\;}^\cup_\cap}^a_a,{{\;}^\cup_\cap}^b_b>$.
\end{enumerate}
\end{lemma}

\begin{proof}
The case $g=ohs$ is discussed in \cite{bbc}, and the case $g=u$ is discussed in \cite{bc}. In the case $g=s'kp$ we can use the following formulae:
\begin{eqnarray*}
{\;}^\cup_\cap
&=&\sum_{ij}e_{ij}\\
|{\:}^\cup_\cap\,|
&=&\sum_ie_{ii}\otimes e_{ii}
\end{eqnarray*}

The commutation conditions $|{\:}^\cup_\cap\,|\in End(u\otimes\bar{u})$ and $|{\:}^\cup_\cap\,|\in End(\bar{u}\otimes u)$ correspond to the cubik condition, and the extra relations ${\;}^\cup_\cap\in End(u)$ and ${\;}^\cup_\cap\in End(\bar{u})$ correspond to the magik condition. Together with the fact that orthogonal plus magik means magic', this gives all the $g=s'kp$ assertions.
\end{proof}

We can color the diagrams in several ways: either by putting the sequence $xyyxxyyx\ldots$ on both rows of points, or by putting $\alpha,\beta$ on both rows, then by replacing $a\to xy,b\to yx$. We say that the diagram is colored if all the strings match, and half-colored, if there is an even number of unmatches.

\begin{theorem}
For $g=ohs^*ukp$ we have $CA_g(n)={\rm span}(D_g)$, where:
\begin{enumerate}
\item $D_s(k,l)$ is the set of all diagrams between $2k$ points and $2l$ points.

\item $D_{s'}(k,l)=D_s(k,l)$ for $k-l$ even, and $D_{s'}(k,l)=\emptyset$ for $k-l$ odd.

\item $D_h(k,l)$ consists of diagrams which are colorable $xyyxxyyx\ldots$.

\item $D_o(k,l)$ is the image of $D_s(k/2,l/2)$ by the doubling map.

\item $D_p(\alpha,\beta)$ consists of diagrams half-colorable $a\to xy, b\to yx$.

\item $D_k(\alpha,\beta)$ consists of diagrams colorable $a\to xy, b\to yx$.

\item $D_u(\alpha,\beta)$ consists of double diagrams, colorable $a\to xy,b\to yx$.
\end{enumerate}
\end{theorem}

\begin{proof}
This is clear from the above lemma, by composing diagrams. The case $g=ohsu$ is discussed in \cite{bbc,bc}, and the case $g=s'kp$ is similar.
\end{proof}

\begin{theorem}
We have the following isomorphisms:
\begin{enumerate}
\item $A_u(n)=\tilde{A}_o(n)$.
\item $A_h(n)=\tilde{A}_k(n)$.
\item $A_p(n)=\tilde{A}_{s^*}(n)$.
\end{enumerate}
\end{theorem}

\begin{proof}
It follows from definitions that we have arrows from left to right. Now since by Theorem 3.3 the spaces $End(u\otimes\bar{u}\otimes u\otimes\ldots)$ are the same at right and at left, Theorem 5.1 in \cite{b2} applies, and gives the arrows from right to left.
\end{proof}

Observe that the assertion (1), known since \cite{b1}, is nothing but the isomorphism $U_n^+=O_n^{+c}$ mentioned in the beginning of the first section.

\section{Freeness, level, doubling}

We use the notion of free Hopf algebra, introduced in \cite{bbc}. Recall that a morphism $(A,u)\to (B,v)$ induces inclusions $Hom(u^\alpha,u^\beta)\subset Hom(v^\alpha,v^\beta)$.

\begin{definition}
A finitely generated Hopf algebra $(A,u)$ is called free if:
\begin{enumerate}
\item The canonical map $A_u(n)\to A_s(n)$ factorizes through $A$.
\item The spaces $Hom(u^\alpha,u^\beta)\subset{\rm span}(D_s(\alpha,\beta))$ are spanned by diagrams. 
\end{enumerate}
\end{definition}

It follows from Theorem 3.3 that the algebras $A_{ohs^*ukp}$ are free.

In the orthogonal case $u=\bar{u}$ we say that $A$ is free orthogonal, and in the general case, we also say that $A$ is free unitary.
 
\begin{theorem}
If $A$ is free orthogonal then $\tilde{A}$ is free unitary.
\end{theorem}

\begin{proof}
It is shown in \cite{b2} that the tensor category of $\tilde{A}$ is generated by the tensor category of $A$, embedded via alternating words, and this gives the result.
\end{proof}

\begin{definition}
The level of a free orthogonal Hopf algebra $(A,u)$ is the smallest number $l\in\{0,1,\ldots,\infty\}$ such that $1\in u^{\otimes 2l+1}$.
\end{definition}

As the level of examples, for $A_s(n)$ we have $l=0$, and for $A_{ohs'}(n)$ we have $l=\infty$. This follows indeed from Theorem 3.3.

\begin{theorem}
If $l<\infty$ then $\tilde{A}=C(\mathbb T)*A$.
\end{theorem}

\begin{proof}
Let $<r>$ be the algebra generated by the coefficients of $r$. From $1\in<u>$ we get $z\in<zu_{ij}>$, hence $<zu_{ij}>=<z,u_{ij}>$, and we are done.
\end{proof}

\begin{corollary}
$A_p(n)=C(\mathbb T)*A_s(n)$.
\end{corollary}

\begin{proof}
For $A_s(n)$ we have $1\in u$, hence $l=0$, and Theorem 4.4 applies.
\end{proof} 

We can define a ``doubling'' operation $A\to A_2$ for free orthogonal algebras, by using Tannakian duality, in the following way: the spaces $Hom(u^k,u^l)$ with $k-l$ even remain by definition the same, and those with $k-l$ odd become by definition empty. The interest in this operation is that $A_2$ has infinite level.

At the level of examples, the doublings are $A_{ohs^*}(n)\to A_{ohs'}(n)$.

\begin{proposition}
For a free orthogonal algebra $A$, the following are equivalent:
\begin{enumerate}
\item $A$ has infinite level.
\item The canonical map $A_2\to A$ is an isomorphism.
\item The quotient map $A\to A_s(n)$ factorizes through $A_{s'}(n)$.
\end{enumerate}
\end{proposition}

\begin{proof}
The equivalence between (1) and (2) is clear from definitions, and the equivalence with (3) follows from Tannakian duality.
\end{proof}

\section{The main result}

We know from Theorem 3.4 that the two rows of the diagram formed by the algebras $A_{ohs^*ukp}$ are related by the operation $A\to\tilde{A}$. Moreover, the results in the previous section suggest that the correct choice in the lower row is $s^*=s'$. The following general result shows that this is indeed the case.

\begin{theorem}
The operation $A\to\tilde{A}$ induces a one-to-one correspondence between the following objects:
\begin{enumerate}
\item Free orthogonal algebras of infinite level.
\item Free unitary algebras satisfying $A=\tilde{A}$.
\end{enumerate}
\end{theorem}

\begin{proof}
We use the notations $\gamma_k=abab\ldots$ and $\delta_k=baba\ldots$ ($k$ terms each).

We know from Theorem 4.2 that the operation $A\to\tilde{A}$ is well-defined, between the algebras in the statement. Moreover, since by Tannakian duality an orthogonal algebra of infinite level is determined by the spaces $Hom(u^k,u^l)$ with $k-l$ even, we get that $A\to\tilde{A}$ is injective, because these spaces are:
$$Hom(u^k,u^l)
=Hom((zu)^{\gamma_k},(zu)^{\gamma_l})$$

It remains to prove surjectivity. So, let $A$ be free unitary satisfying $A=\tilde{A}$. We have $CA={\rm span}(D)$ for certain sets of diagrams $D(\alpha,\beta)\subset D_s(\alpha,\beta)$, so we can define a collection of sets $D_2(k,l)\subset D_s(k,l)$ in the following way:
\begin{enumerate}
\item For $k-l$ even we let $D_2(k,l)=D(\gamma_k,\gamma_l)$.
\item For $k-l$ odd we let $D_2(k,l)=\emptyset$.
\end{enumerate}

It follows from definitions that $C_2={\rm span}(D_2)$ is a category, with duality and involution. We claim that $C_2$ is stable under $\otimes$. Indeed, for $k,l$ even we have:
\begin{eqnarray*}
D_2(k,l)\otimes D_2(p,q)
&=&D(\gamma_k,\gamma_l)\otimes D(\gamma_p,\gamma_q)\\
&\subset&D(\gamma_k\gamma_p,\gamma_l\gamma_q)\\
&=&D(\gamma_{k+p},\gamma_{l+q})\\
&=&D_2(k+p,l+q)
\end{eqnarray*}

For $k,l$ odd and $p,q$ even, we can use the canonical antilinear isomorphisms $Hom(u^{\gamma_K},u^{\gamma_L})\simeq Hom(u^{\delta_K},u^{\delta_L})$, with $K,L$ odd. At the level of diagrams we get equalities $D(\gamma_k,\gamma_L)=D(\delta_K,\delta_L)$, that can be used in the following way:
\begin{eqnarray*}
D_2(k,l)\otimes D_2(p,q)
&=&D(\delta_k,\delta_l)\otimes D(\gamma_p,\gamma_q)\\
&\subset&D(\delta_k\gamma_p,\delta_l\gamma_q)\\
&=&D(\delta_{k+p},\delta_{l+q})\\
&=&D_2(k+p,l+q)
\end{eqnarray*}

Finally, for $k,l$ odd and $p,q$ odd, we can proceed as follows:
\begin{eqnarray*}
D_2(k,l)\otimes D_2(p,q)
&=&D(\gamma_k,\gamma_l)\otimes D(\delta_p,\delta_q)\\
&\subset&D(\gamma_k\delta_p,\gamma_l\delta_q)\\
&=&D(\gamma_{k+p},\gamma_{l+q})\\
&=&D_2(k+p,l+q)
\end{eqnarray*}

Thus we have a Tannakian category, and by Woronowicz's results in \cite{wo2} we get an algebra $A_2$. This algebra is free orthogonal, of infinite level. Moreover, the spaces $End(u\otimes\bar{u}\otimes u\otimes\ldots)$ being the same for $A$ and $A_2$, Theorem 5.1 in \cite{b2} applies, and gives $\tilde{A_2}=\tilde{A}$. Now since we have $A=\tilde{A}$, we are done.
\end{proof}

\end{document}